\documentclass[12pt]{article}
\usepackage{amssymb,amsthm,amsmath,latexsym}
\usepackage{url}
\usepackage{color}
\usepackage{comment}
\usepackage{lscape}
\newtheorem{thm}{Theorem}
\newtheorem{prop}[thm]{Proposition}
\newtheorem{lem}[thm]{Lemma}

\theoremstyle{remark}
\newtheorem{rem}[thm]{Remark}
\newcommand{\FF}{\mathbb{F}}
\newcommand{\ZZ}{\mathbb{Z}}

\newcommand{\1}{\mathbf{1}}

\DeclareMathOperator{\wt}{wt}

\DeclareMathOperator{\rank}{rank}

\begin{document}
\title{Hadamard matrices related to
a certain series of ternary self-dual codes
}

\author{
Makoto Araya\thanks{Department of Computer Science,
Shizuoka University,
Hamamatsu 432--8011, Japan.
email: {\tt araya@inf.shizuoka.ac.jp}},
Masaaki Harada\thanks{
Research Center for Pure and Applied Mathematics,
Graduate School of Information Sciences,
Tohoku University, Sendai 980--8579, Japan.
email: \texttt{mharada@tohoku.ac.jp}}
and
Koji Momihara\thanks{
Division of Natural Science, Faculty of Advanced Science and Technology, 
Kumamoto University, Kumamoto 860--8555, Japan.
email: \texttt{momihara@educ.kumamoto-u.ac.jp}.}
}

\maketitle

\begin{abstract}
In 2013, Nebe and Villar gave a series of ternary self-dual codes of
length $2(p+1)$ for a prime $p$ congruent to $5$ modulo $8$.
As a consequence, the third ternary extremal self-dual code
of length $60$ was found.
We show that the ternary self-dual code contains codewords which form
a Hadamard matrix of order $2(p+1)$
when $p$ is congruent to $5$ modulo $24$.
In addition, it is shown that
the ternary self-dual code is generated by the rows of 
the  Hadamard matrix.
We also demonstrate that the third ternary extremal self-dual code of length $60$ 
contains at least two inequivalent Hadamard matrices.
\end{abstract}

\section{Introduction}\label{sec:Intro}

Self-dual codes are one of the most interesting classes of codes.
This interest is justified by many combinatorial objects
and algebraic objects related to self-dual codes
(see e.g., \cite{RS-Handbook}).
A Hadamard matrix is a kind of orthogonal matrix appearing in many 
research areas of Mathematics and practical applications (see e.g., \cite{Seberry}
and~\cite{SeberryYamada}). One of the 
interesting and successful applications of Hadamard matrices is their 
use as codes. In particular, a special class of Hadamard matrices can 
give rise to self-dual codes as their row spaces.
In this paper, we are interested in Hadamard matrices related to ternary 
self-dual codes found by Nebe and Villar~\cite{NV}.

A ternary self-dual code $C$ of length $n$
is an $[n,n/2]$ code over the finite field of
order $3$ satisfying
$C=C^\perp$, where $C^\perp$ is the dual code  of $C$.
A ternary self-dual code  of length $n$ exists if and only if $n$ is divisible by four.
It was shown in~\cite{MS-bound} that
the minimum weight $d$ of a ternary self-dual code of length $n$
is bounded by $d\leq 3 \lfloor n/12 \rfloor+3$.
If $d=3\lfloor n/12 \rfloor+3$, then the code is called {extremal}.
For $n \in \{4,8,12,\ldots,64\}$,
it is known that there is a ternary extremal self-dual code of length $n$ (see~\cite[Table~6]{Huffman}).
The ternary extended quadratic residue codes and 
the Pless symmetry codes are well known families of ternary (self-dual) codes.
It is known that
the ternary extended quadratic residue code $QR_{60}$ of length $60$ and 
the Pless symmetry code $P_{60}$ of length $60$ 
are ternary extremal self-dual codes (see~\cite[Table~XII]{RS-Handbook}).
In 2013, Nebe and Villar~\cite{NV} gave a series of ternary self-dual codes 
of length $2(p + 1)$ for all primes $p \equiv 5 \pmod 8$.
As a consequence, the third ternary extremal self-dual code
of length $60$ was found.

A {Hadamard} matrix $H$ of order $n$ is an $n \times n$ matrix
whose entries are from $\{ 1,-1 \}$ such that $H H^T = nI_n$,
where $H^T$ is the transpose of $H$ and $I_n$ is the identity matrix of order $n$.
It is known that
the order $n$ is necessarily $1,2$, or a multiple of $4$.
Recently, Tonchev~\cite{T} studied Hadamard matrices of order $n$
formed by codewords of weight $n$ in ternary extremal self-dual codes of length $n$,
especially the extended quadratic residue codes and
the Pless symmetry codes.
{From} the construction, the extended quadratic residue code contains a   type I
Paley-Hadamard matrix.
The Pless symmetry code contains a type II Paley-Hadamard matrix~\cite{Pless}.
Tonchev~\cite{T} showed that the Pless symmetry code of length $36$ contains exactly
two inequivalent Hadamard matrices of order $36$.
This motivates us to study the existence of 
Hadamard matrices of order $n$ 
formed by codewords of weight $n$ in
ternary self-dual codes found by Nebe and Villar~\cite{NV}.

The paper is organized as follows.
In Section~\ref{Sec:2},
definitions, notations and basic results are given.
Especially, we review the construction of ternary self-dual codes $NV^{(a)}(p)$
in~\cite{NV} of length $2(p+1)$, where $p$ is a prime with
$p \equiv 5 \pmod{8}$ and $a \in \{1,-1\}$.
In Section~\ref{Sec:H}, we show that
$NV^{(a)}(p)$ contains $2(p+1)$ codewords of weight $2(p+1)$ which form
a Hadamard matrix $H_{NV^{(a)}(p)}$ of order $2(p+1)$ for any prime
$p \equiv 5 \pmod{24}$ and $a \in \{1,-1\}$ (see 
Theorem~\ref{thm}, which is our main theorem of this paper).
We also give characterizations of the Hadamard matrices $H_{NV^{(a)}(p)}$ of order $2(p+1)$.
In particular, it is shown that
the ternary self-dual code $NV^{(a)}(p)$ is generated by the rows of 
the Hadamard matrix $H_{NV^{(a)}(p)}$.
This gives an alternative construction of the ternary self-dual code $NV^{(a)}(p)$.
By Theorem~\ref{thm}, 
the third ternary extremal self-dual code ${NV^{(1)}(29)}$
of length $60$, which was found in~\cite{NV}, 
contains a Hadamard matrix of order $60$.
In Section~\ref{Sec:60}, our computer search shows that $NV^{(1)}(29)$ contains one more 
Hadamard matrix of order $60$.
Finally, in Section~\ref{Sec:4nega}, we demonstrate that the currently known three
ternary extremal
self-dual codes of length $60$ are constructed as four-negacirculant codes.

\section{Preliminaries}\label{Sec:2}

In this section, we give definitions and some known results
of ternary self-dual codes and Hadamard matrices used in this paper.
Especially, we give details for the construction of ternary self-dual codes $NV^{(a)}(p)$
in~\cite{NV} of length $2(p+1)$, where $p$ is a prime with
$p \equiv 5 \pmod{8}$ and $a \in \{1,-1\}$.

\subsection{Ternary self-dual codes}
Let $\FF_3=\{0,1,2\}$ denote the finite field of order $3$.
A ternary $[n,k]$ \emph{code} $C$ is a $k$-dimensional vector subspace
of $\FF_3^n$.
All codes in this paper are ternary.
The parameter $n$ is called the \emph{length} of $C$.
A \emph{generator matrix} of $C$
is a $k \times n$ matrix whose rows are a basis of $C$.
The \emph{weight} $\wt(x)$ of a vector $x$ of $\FF_3^n$ is 
the number of non-zero components of $x$.
A vector of $C$ is called a \emph{codeword}.
The minimum non-zero weight of all codewords in $C$ is called
the \emph{minimum weight} of $C$.
The \emph{weight enumerator} of $C$ is given by $\sum_{c \in C} y^{\wt(c)}
\in \ZZ[y]$.

The {\em dual} code $C^{\perp}$ of a ternary code
$C$ of length $n$ is defined as
$
C^{\perp}=
\{x \in \FF_3^n \mid x \cdot y = 0 \text{ for all } y \in C\},
$
where $x \cdot y$ is the standard inner product.
A ternary code $C$ is \emph{self-dual} if $C=C^\perp$.
A ternary self-dual code  of length $n$ exists if and only if $n$ is divisible by four.
Two ternary codes $C$ and $C'$ are \emph{equivalent} if there is a
monomial matrix $P$ over $\mathbb{F}_3$ with $C' = C \cdot P$,
where $C \cdot P = \{ x P\mid  x \in C\}$.
We denote two equivalent ternary codes $C$ and $D$ by $C \cong D$.
All ternary self-dual codes were classified in~\cite{CPS}, \cite{HM}, \cite{MPS} 
and~\cite{PSW} for lengths up to $24$.
 
\subsection{Ternary extremal self-dual codes}

It was shown in~\cite{MS-bound} that
the minimum weight $d$ of a ternary self-dual code of length $n$
is bounded by $d\leq 3 \lfloor n/12 \rfloor+3$.
If $d=3\lfloor n/12 \rfloor+3$, then the code is called \emph{extremal}.
For $n \in \{4,8,12,\ldots,64\}$,
it is known that there is a ternary extremal self-dual code of length $n$ (see~\cite[Table~6]{Huffman}).
By the Assmus--Mattson theorem~\cite{AM}, the supports of codewords of minimum
weight in a ternary extremal self-dual code of length divisible by $12$
form a $5$-design.
This is a reason for our interest in ternary extremal self-dual codes of length divisible by $12$.

The weight enumerator of a ternary extremal self-dual code of length $n$ is uniquely
determined for each $n$~\cite{MS-bound}.
The number $A_n$ of codewords of weight $n$ in a 
ternary extremal self-dual code of length $n$ is listed in Table~\ref{Tab:An} for
$n=12,24,36,48,60$ (see~\cite{T}).
Note that $A_n=2n$ for $n=12,24,48$.

\begin{table}[thb]
\caption{Numbers $A_n$ of codewords of weight $n$}
\label{Tab:An}
\centering
\medskip
{\small
\begin{tabular}{c|ccccc}
\noalign{\hrule height0.8pt}
$n$       &12 &24 &36 &48 &60 \\
\hline
$A_n$   &
 24 
& 48
& 888
& 96
& 41184\\
\noalign{\hrule height0.8pt}
\end{tabular}
}
\end{table}

The ternary extended quadratic residue codes and 
the Pless symmetry codes are well known families of ternary (self-dual) codes.
More precisely,
the extended quadratic residue code $QR_{p+1}$ of length $p+1$
is a ternary self-dual code when 
$p$ is a prime such that $p \equiv -1 \pmod{12}$ (see~\cite[Chapter~6]{HP}).
The Pless symmetry code $P_{2q+2}$ of length $2q+2$ is
a ternary self-dual code when $q$ is a prime power
such that $q \equiv -1 \pmod 6$~\cite{Pless} (see also~\cite[Chapter~10]{HP}).
The extended quadratic residue codes $QR_{n}$ and the Pless symmetry codes 
$P_{n}$ yield ternary extremal self-dual codes when $n \le 60$  
(see~\cite{RS-Handbook}).
More precisely, 
$P_{36}$ is the currently known ternary extremal self-dual code of length $36$,
$QR_{48}$ and $P_{48}$ are the currently known ternary extremal self-dual codes
of length $48$.
In addition, 
$QR_{60}$ and $P_{60}$ are ternary extremal self-dual codes of length $60$.

\subsection{Ternary self-dual codes given in~\cite{NV}}
\label{sec:NV}

In 2013, Nebe and Villar~\cite{NV} gave a new series of ternary self-dual codes $NV^{(a)}(p)$
of length $2(p + 1)$ for all primes $p \equiv 5 \pmod 8$ and $a \in \{1,-1\}$
(see also~\cite[Section~4]{BCV} for the details).
Here, we review the construction of the ternary self-dual codes $NV^{(a)}(p)$.

Suppose that $p \equiv 5 \pmod 8$.
Let $\FF_p=\{0,1,\ldots,p-1\}$ denote the finite field of order $p$.
Let $\chi$ denote the quadratic character of $\FF_p$.
Define two $p \times p$ matrices $R_X=({r_X}_{a,b})$ and $R_Y=({r_Y}_{a,b})$ as follows
\begin{align*}
{r_X}_{a,b}=&
\begin{cases}
0, & \text{ if }a=b \text{ or }b-a \text{ is not a nonzero square  in } \FF_p,\\
\chi(c), & \text{ if }b-a \text{ is a nonzero square $c^2$ in } \FF_p,\\
\end{cases}
\\
{r_Y}_{a,b}=&
\begin{cases}
0, & \text{ if }a=b \text{ or }2(b-a) \text{ is not a nonzero square  in }\FF_p,\\
\chi(c), & \text{ if }2(b-a) \text{ is a nonzero square $c^2$ in } \FF_p,\\
\end{cases}
\end{align*}
where rows and columns of $R_X$ and $R_Y$ are indexed by the elements of
$\FF_p$ with a fixed ordering.
Then 
define two $(p+1) \times (p+1)$ matrices $X$ and $Y$ as follows
\[
X=
\left(
\begin{array}{cccc}
0 & 1 & \cdots & 1 \\
-1 & &&\\
\vdots & & R_X& \\
-1 & &&\\
\end{array}
\right)
\text{ and }
Y=
\left(
\begin{array}{cccc}
0 & 0 & \cdots & 0 \\
0 & &&\\
\vdots & & R_Y& \\
0 & &&\\
\end{array}
\right).
\]
In addition, define two $2(p+1) \times 2(p+1)$ matrices as follows
\[
B_w=
\left(
\begin{array}{rr}
X & Y \\
-Y^T & X^T
\end{array}
\right) \text{ and }
B_{\epsilon w}=
\left(
\begin{array}{rr}
-Y^T & X^T\\
-X & -Y 
\end{array}
\right).
\]
Throughout this paper,
let  $I_n$ denote the identity matrix of order $n$.
For $a=1$ and $-1$, 
let $NV^{(a)}(p)$ denote the ternary code generated by the matrix $M$,
where
\[
M=
\begin{cases}
a I_{2(p+1)} + B_w, &\text{ if } p \equiv 5 \pmod{24}, \\
a I_{2(p+1)} + B_w +B_{\epsilon w}, &\text{ if } p \equiv 13 \pmod{24}.
\end{cases}
\]
Then $NV^{(a)}(p)$ is self-dual~\cite{NV} (see also~\cite[Theorem~8]{BCV}).

\begin{prop}[Nebe and Villar~\cite{NV}]
$NV^{(1)}(29) \cong NV^{(-1)}(29)$ and 
$ QR_{60}\not \cong NV^{(1)}(29) \not \cong P_{60}$.
\end{prop}

The above proposition means that $NV^{(1)}(29)$ is the third ternary extremal self-dual code
of length $60$.
In this paper, we denote the  code  $NV^{(1)}(29)$  by $NV_{60}$.

\subsection{Hadamard matrices and results in~\cite{T}}
A \emph{Hadamard} matrix $H$ of order $n$ is an $n \times n$ matrix
whose entries are from $\{ 1,-1 \}$ such that $H H^T = nI_n$,
where $H^T$ is the transpose of $H$.
It is known that
the order $n$ is necessarily $1,2$, or a multiple of $4$.
A Hadamard matrix $H$  of order $n$ is called \emph{skew} if $H=A+I_n$, where
$A=-A^T$.
Two Hadamard matrices $H$ and $K$ are said to be \emph{equivalent}
if there is $(1,-1,0)$-monomial matrices $P$ and $Q$ with
$K = PHQ$.
An \emph{automorphism} of a Hadamard matrix $H$ is
an equivalence of $H$ to itself, i.e., a pair $(P,Q)$ of monomial
matrices $P$ and $Q$ such that $H = PHQ$.
The set of all automorphisms of $H$ forms a group, 
called the \emph{automorphism group} of $H$, 
under the component-wise product: $(P_1,Q_1)(P_2,Q_2)=(P_1P_2,Q_1Q_2)$.  
%

Recently, Tonchev~\cite{T} studied Hadamard matrices of order $n$
formed by codewords of weight $n$ in ternary extremal self-dual codes of length $n$,
especially the extended quadratic residue codes and
the Pless symmetry codes.
In the context of Hadamard matrices, we consider the element $0,1,2$ of $\FF_3$ as 
$0,1,-1$ of $\ZZ$, throughout this paper.
It is trivial that $n \equiv 0 \pmod{12}$ if a  ternary (extremal) self-dual code of length $n$
contains a Hadamard matrix formed by codewords of weight $n$.
This is another reason for our interest in ternary extremal self-dual codes of length divisible by $12$.

{From} the construction, the extended quadratic residue code contains a   type I
Paley-Hadamard matrix.
The Pless symmetry code contains a type II Paley-Hadamard matrix~\cite{Pless}.
Tonchev~\cite{T} showed that $P_{36}$ contains exactly
two inequivalent Hadamard matrices of order $36$.
In addition, Tonchev~\cite{T} gave a natural question, namely, 
is there any other ternary extremal self-dual code of length $36$, $48$, or $60$ 
which contains a Hadamard matrix?
This motivates us to study the existence of 
Hadamard matrices of order $2(p+1)$ 
formed by codewords of weight $2(p+1)$ in
the ternary self-dual codes $NV^{(a)}(p)$
found by Nebe and Villar~\cite{NV}.

\section{Hadamard matrices related to $NV^{(a)}(p)$ }
\label{Sec:H}

Throughout this section,
suppose that $p$ is a prime with $p \equiv 5 \pmod{24}$.
In this section, we show that
$NV^{(a)}(p)$ contains $2(p+1)$ codewords of weight $2(p+1)$ which form
a Hadamard matrix $H_{NV^{(a)}(p)}$ of order $2(p+1)$ for
$a \in \{1,-1\}$.
We also give characterizations of the Hadamard matrices $H_{NV^{(a)}(p)}$ of order $2(p+1)$.

Let $X$ and $Y$ be the $(p+1) \times (p+1)$ matrices as defined  in Section~\ref{sec:NV}.
As described there, 
for $a=1$ and $-1$,
the ternary code $NV^{(a)}(p)$ generated by the following matrix
\[
a I_{2(p+1)} + 
\left(
\begin{array}{rr}
X & Y \\
-Y^T & X^T
\end{array}
\right)
\]
is a ternary self-dual code~\cite{NV} (see also~\cite[Theorem~8]{BCV}).

The following is our main theorem of this paper.

\begin{thm}\label{thm}
Suppose that $p \equiv 5 \pmod{24}$ and $a \in \{1,-1\}$.
Then the ternary self-dual code $NV^{(a)}(p)$ of length $2(p+1)$
contains $2(p+1)$ codewords of weight $2(p+1)$ which form
a Hadamard matrix of order $2(p+1)$.
\end{thm}
\begin{proof}
Since $NV^{(a)}(p)$ is generated by the following matrix
\[
\left(
\begin{array}{cc}
X+aI_{p+1} & Y \\
-Y^T & X^T+aI_{p+1}
\end{array}
\right),
\]
the rows of the following two matrices
\begin{multline*}
\left(
\begin{array}{cc}
X-Y^T+aI_{p+1} & Y + X^T +aI _{p+1}
\end{array}
\right)
\text{ and }
\\
\left(
\begin{array}{cc}
-Y^T-X-aI_{p+1} & X^T-Y+aI_{p+1}
\end{array}
\right)
\end{multline*}
are codewords of $NV^{(a)}(p)$.
{From} the definition of $X$ and $Y$, the $2(p+1)$ codewords has weight $2(p+1)$.
In addition, we regard the following matrix as a $\ZZ$-matrix
\begin{equation}\label{eq:H}
H_{NV^{(a)}(p)}=
\left(
\begin{array}{cc}
X-Y^T+aI_{p+1} & Y + X^T +aI_{p+1} \\
-Y^T-X-aI_{p+1} & X^T-Y+aI_{p+1}
\end{array}
\right).
\end{equation}
Since it is known~\cite{BCV} that
\[
X^T=-X, Y^T=-Y, XY=YX \text{ and } X^2+Y^2=-pI_{p+1},
\]
$H_{NV^{(a)}(p)}$ is a Hadamard matrix of order $2(p+1)$.
\end{proof}

Now we give characterizations of  Hadamard matrices $H_{NV^{(a)}(p)}$ of order $2(p+1)$.

\begin{prop}
Let $H_{NV^{(a)}(p)}$ denote the Hadamard matrix given in~\eqref{eq:H}.
Then $a H_{NV^{(a)}(p)}$ is a skew Hadamard matrix for $a=1$ and $-1$.
\end{prop}
\begin{proof}
The claim follows from that $H_{NV^{(a)}(p)} +H_{NV^{(a)}(p)}^T=2aI_{2(p+1)}$.
\end{proof}

Note that 
$H_{NV^{(a)}(p)}$ has the form
$
H_{NV^{(a)}(p)}=
\left(
\begin{array}{cc}
A & B \\
-B^T & A^T
\end{array}
\right)$,
where
\begin{align*}
A=&X-Y^T+aI_{p+1}
=
\left(
\begin{array}{cccc}
a& 1 & \cdots & 1\\
-1   & &&\\
\vdots& &  A' & \\
-1   &&&\\
\end{array}
\right) \text{ and}\\
B=&Y + X^T +aI_{p+1}
=
\left(
\begin{array}{cccc}
a& -1 & \cdots & -1\\
1   & &&\\
\vdots& &  B' & \\
1   &&&\\
\end{array}
\right),
\end{align*}
for some $p\times p$ matrices $A'$ and $B'$. 
Let $\omega$ be a fixed primitive element of $\mathbb{F}_p$. Define 
\begin{equation}\label{eq:cosetsCi}
C_i=\omega^i\langle \omega^4\rangle, i=0,1,2,3, 
\end{equation}
which are  the cosets of the multiplicative subgroup of index $4$ of $\mathbb{F}_p$. 
Note that $C_1$ and $C_3$ are interchanged if we choose $\omega^{-1}$ as a primitive element instead of $\omega$.  
Since $p\equiv 5\pmod{8}$, we have $-1\in C_2$ and 
$2\in NQ=C_1\cup C_3$, where $NQ$ denotes the set of nonsquares 
in ${\mathbb F}_p$.
%
Hence, $-2\in C_1$ or $-2\in C_3$, depending on the choice of $\omega$.
We denote by $C_\epsilon$ the coset containing $-2$, that is,
\begin{equation}\label{eq:ep}
-2 \in C_\epsilon.
\end{equation}
Then $A'=(a'_{s,t})$ and $B'=(b'_{s,t})$ can be written as 
\begin{equation}\label{eq:AB}
\begin{split}
a'_{s,t}&=\begin{cases}
a, &\text{ if $t=s$, }\\
1, &\text{ if $t-s\in C_0\cup C_{\epsilon}$, }\\
-1, &\text{ if $t-s\in C_2\cup C_{\epsilon+2}$, }
\end{cases}
\\
b'_{s,t}&=\begin{cases}
a, &\text{ if $t=s$, }\\
1, &\text{ if $t-s\in C_2\cup C_{\epsilon}$, }\\
-1, &\text{ if $t-s\in C_0\cup C_{\epsilon+2}$.}
\end{cases}
\end{split}
\end{equation}
Throughout this section, we reduce the subscript of  $C_{i}$ modulo $4$.

Let $G$ be an additively written abelian group of order $v$.  Two subsets  $D_1$ and $D_2$ of $G$ with $k=|D_1|=|D_2|$ are called {\it $(v,k,\lambda)$ supplementary difference sets} if 
the list of differences $x-y$, $x,y \in D_i$, $i=1,2$, represents every nonzero element of $G$ exactly $\lambda$ times.
Fixing an ordering for the elements of $G$, we define a matrix $M=(m_{i,j})$ by   
\[
m_{i,j}=\begin{cases}
1,& \text{ if }  j-i\in X, \\
-1, & \text{ if } j-i\not\in X,
\end{cases}
\]
for $X\subset G$. 
The matrix $M$ is called a \emph{type-1} matrix of $X$.

The following construction of Hadamard matrices easily 
follows from~\cite[Corollary~4.5~(i) and Lemma~4.8]{Seberry}.

\begin{lem}\label{lem:SDS}
Let $D_i $, $i=1,2$, be $(2m+1,m,m-1)$ supplementary difference sets in an abelian group $G$ of order $v=2m+1$.  
Furthermore, let  
$M_1$ (resp.\ $M_2$) be the type-1 matrix of $D_1$ (resp.\ $D_2$). 
Then
\[
H(D_1,D_2)=\begin{pmatrix} 
1  &1 &  \1_v&  -\1_v  \\
-1  & 1 & -\1_v& -\1_v \\
-\1_v^T  &  \1_v^T &-M_1 & -M_2 \\
\1_v ^T & \1_v^T&M_2^T & -M_1^T
 \end{pmatrix}
\]
is a Hadamard matrix of order $4(m+1)$,
where $\1_v$ denotes the all-one vector of length $v$.
\end{lem}

It is known~\cite{S} that for any prime $p\equiv 5\pmod{24}$
and any $i=0,1,2,3$, 
the sets $C_i\cup C_{i+1}$ and $C_{i+1}\cup C_{i+2}$ are $(p,(p-1)/2,(p-3)/2)$ 
supplementary difference sets in the additive group of $\FF_p$, where $C_i$'s are defined 
as in~\eqref{eq:cosetsCi}
(more generally, the same claim holds for  any prime power $p\equiv 5\pmod{8}$).
By Lemma~\ref{lem:SDS}, 
$H(C_i\cup C_{i+1},C_{i+1}\cup C_{i+2})$ is a Hadamard matrix of order $2(p+1)$.
%
%
Furthermore, from~\eqref{eq:AB}, the following theorem  holds. 
\begin{thm}
Let $H_{NV^{(a)}(p)}$ denote the Hadamard matrix defined as in~\eqref{eq:H}.
Then
$H_{NV^{(1)}(p)}$ and $H_{NV^{(-1)}(p)}$ are equivalent to $H(C_2\cup C_{2+\epsilon},C_0\cup C_{\epsilon+2})$ and  $H(C_0\cup C_{\epsilon},C_2\cup C_{\epsilon})$, respectively,
where $C_\epsilon$ is defined as in~\eqref{eq:ep}.
\end{thm}

Although the proof of the following proposition is somewhat trivial,
we give it for the sake of completeness.

\begin{prop}
The ternary self-dual code $NV^{(a)}(p)$ is generated by the rows of 
the Hadamard matrix $H_{NV^{(a)}(p)}$ defined as in~\eqref{eq:H}.
\end{prop}
\begin{proof}
It is sufficient to show that  $\rank_3(H_{NV^{(a)}(p)})=p+1$. 
Since $H_{NV^{(a)}(p)} +H_{NV^{(a)}(p)}^T=2aI_{2(p+1)}$,
we have
\begin{align*}
2(p+1)&=\rank_3(2aI_{2(p+1)})=\rank_3(H_{NV^{(a)}(p)} +H_{NV^{(a)}(p)}^T)\\
&\le 2\rank_3(H_{NV^{(a)}(p)}), 
\end{align*}  
i.e., $p+1\le \rank_3(H_{NV^{(a)}(p)})$. 
On the other hand, since $p+1\equiv 0\pmod{3}$, 
$H_{NV^{(a)}(p)}H_{NV^{(a)}(p)}^T\equiv O\pmod{3}$, 
where $O$ denotes the $2(p+1) \times 2(p+1)$ zero matrix.
This implies that  $\rank_3(H_{NV^{(a)}(p)})\le p+1$.
This completes the proof. 
\end{proof}

The above proposition gives an alternative construction of the ternary self-dual code $NV^{(a)}(p)$.

\section{Hadamard matrices related to $NV_{60}$ }
\label{Sec:60}

By Theorem~\ref{thm}, the third ternary extremal self-dual code
$NV_{60}$ of length $60$, which was found in~\cite{NV}, 
contains a Hadamard matrix of order $60$.
In this section, our computer search found one more Hadamard matrix of order $60$
in $NV_{60}$. 
All computer calculations in this section were done by programs in the language C 
and programs in \textsc{Magma}~\cite{Magma}.

Any ternary extremal self-dual code of length $60$ contains
$41184$ codewords of weight $60$ (see Table~\ref{Tab:An}).
Let $W_{60}$ be the set of $41184$ codewords of weight $60$ in $NV_{60}$.
It is trivial that 
there is a set $W^+_{60}$ consisting of $20592$ codewords of weight $60$ such that 
\[
W_{60}=W^+_{60} \cup \{2x \mid x \in W^+_{60}\}.
\] 
Let $\rho$ be a map from $\FF_3$ to $\ZZ$ sending $0, 1, 2$
to $0, 1, -1$, respectively.
Define the following set
\[
W^\ZZ_{60}=\{(\rho(x_1),\rho(x_2),\ldots,\rho(x_{60})) \mid
(x_1,x_2,\ldots,x_{60}) \in W^+_{60}
\}  \subset \ZZ^{60}.
\]
Then we define the simple undirected graph $\Gamma$,
whose set of vertices is the set $W^\ZZ_{60}$
and two vertices $x$ and $y$ are adjacent
if $x$ and $y$ are orthogonal, noting that $x,y \in \ZZ^{60}$.
Clearly, a $60$-clique in $\Gamma$ gives a Hadamard matrix.
In addition, in order to find a Hadamard matrix, it is sufficient to consider
only $W^\ZZ_{60}$ as the set of vertices of $\Gamma$.
Due to the computational complexity,
by the above approach, our computer search was able to find two $60$-cliques, 
which imply two inequivalent Hadamard matrices $H_{NV,1}$ and $H_{NV,2}$.
The computation for finding cliques
was performed using the clique finding algorithm {\sc Cliquer}~\cite{Cliquer}.
The computation for verifying the inequivalence of $H_{NV,1}$ and $H_{NV,2}$
was done by the \textsc{Magma} function
\texttt{IsHadamardEquivalent}.
Therefore, we have the following proposition.

\begin{prop}
The third ternary extremal self-dual code
$NV_{60}$ of length $60$ contains at least two inequivalent Hadamard matrices of
order $60$ having as rows codewords of weight $60$.
\end{prop}

We verified by \textsc{Magma} that 
$H_{NV,1}$ and $H_{NV^{(1)}(29)}$ are equivalent, and
$H_{NV,1}$ and $H_{NV,2}$ have automorphism groups of
orders $24360$ and $812$, respectively.
These were done by the \textsc{Magma} functions
\texttt{IsHadamardEquivalent} and
\texttt{HadamardAutomorphismGroup}, respectively.

Now we display  the Hadamard matrix $H_{NV,2}$.
Here, instead of this matrix, 
we display its binary Hadamard matrix 
$B_{NV,2}=(H_{NV,2}+J)/2$,
where $J$ is the $60 \times 60$ all-one matrix.
Let $r_i$ denote the $i$-th row of $B_{NV,2}$.
To save space, the vectors $r_1,r_2,\ldots,r_{60}$
are written in octal using $0=(0,0,0)$, $1=(0,0,1),\ldots,7=(1,1,1)$
in Figure~\ref{Fig:2}.
For example, the first row of $H_{NV,2}$
\begin{align*}
&(1,1,1,1,1,1,1,1,1,1,1,1,1,1,1,1,1,1,1,1,1,1,1,1,1,1,1,1,1,1,\\
&-1,1,1,1,1,1,1,1,1,1,1,1,1,1,1,1,1,1,1,1,1,1,1,1,1,1,1,1,1,1)
\end{align*}
corresponds to $77777777773777777777$.

\begin{figure}[htb]
\centering
{\footnotesize
\begin{align*}
777777777737777777770000000000377777777777101523240306056672\\
762032465106141355640221363475131260237102747172112540476254\\
744065152314302733500427471721225404762504427471722625404762\\
724762032427206141357246517440135564143005716364221301174531\\
710152324730605667200717211057047625453007504427473712625404\\
104427471712625404761057163642053011745311057163641453011745\\
651744065116414302736515237101273350306064762032463206141355\\
646517440635564143021363475044260237126263710152321030605667\\
632476203226720614136203246517214135564116364221361174531260\\
171636422130117453121721105716362545301160324651741413556414\\
574406515201430273352110571636254530117421363475041260237126\\
221363475031260237125523710152335030605655232476200566720614\\
234750442702371262542364221363174531260224221363470531260237\\
523710152335030605665232476203166720614151744065152414302733\\
515237101533350306052747172110140476254550152324760605667206\\
476203246520614135564651744065156414302746515237100273350306\\
316364221301174531263172110571076254530132110571632254530117\\
440651523703027335034324651744013556414334717211050047625453\\
347504427423712625403504427471312625404741523247622056672061\\
363475044220237126253642213634345312602340651523713027335030
\end{align*}
\caption{Binary Hadamard matrix $B_{NV,2}$}
\label{Fig:2}
}
\end{figure}

It is worthwhile to determine whether $C$ contains a Hadamard matrix
which is not equivalent to a Paley-Hadamard matrix
for $C=QR_{60}$ and $P_{60}$.

\section{Four-negacirculant codes of length $60$}
\label{Sec:4nega}

In this section, we demonstrate that the currently known ternary extremal
self-dual codes of length $60$ are constructed as four-negacirculant codes.
All computer calculations in this section were done by 
programs in \textsc{Magma}~\cite{Magma}.

An $n \times n$ \emph{negacirculant} matrix has the following form
\[
\left( \begin{array}{cccccc}
r_0&r_1&r_2& \cdots &r_{n-2} &r_{n-1}\\
2r_{n-1}&r_0&r_1& \cdots &r_{n-3}&r_{n-2} \\
2r_{n-2}&2r_{n-1}&r_0& \cdots &r_{n-4}&r_{n-3} \\
\vdots &\vdots & \vdots &&\vdots& \vdots\\
2r_1&2r_2&2r_3& \cdots&2r_{n-1}&r_0
\end{array}
\right).
\]
Let $A$ and $B$ be $n \times n$ negacirculant matrices.
A ternary $[4n,2n]$ code having the following generator matrix
\begin{equation} \label{eq:4}
\left(
\begin{array}{ccc@{}c}
\quad & {\Large I_{2n}} & \quad &
\begin{array}{cc}
A & B \\
2B^T & A^T
\end{array}
\end{array}
\right)
\end{equation}
is called a \emph{four-negacirculant} code.
Many ternary extremal self-dual four-negacirculant codes are known
(see e.g., \cite{HHKK}).

Let $C_1, C_2$ and $C_3$ be the ternary four-negacirculant codes of length $60$,
having generator matrices of form~\eqref{eq:4}, where
the pairs $(r_A,r_B)$
of the first rows $r_A$ and $r_B$ of the negacirculant matrices $A$ and $B$
are as follows
\begin{align*}
&((1, 1, 0, 2, 1, 1, 1, 2, 2, 2, 0, 1, 0, 0, 2), (2, 0, 0, 2, 1, 0, 0, 1, 2, 2, 0, 1, 0, 2, 2)),\\
&((1, 1, 2, 2, 1, 2, 2, 1, 1, 1, 2, 1, 2, 1, 2), (2, 2, 1, 2, 2, 0, 2, 2, 1, 2, 2, 2, 2, 1, 1)),\\
&((1, 0, 0, 1, 1, 2, 2, 0, 2, 1, 1, 0, 0, 0, 2), (1, 2, 0, 0, 2, 2, 1, 1, 0, 0, 0, 0, 2, 2, 0)),
\end{align*}
respectively.
We verified by \textsc{Magma} that  
$C_1 \cong QR_{60}$, $C_2 \cong P_{60}$ and $C_3 \cong NV_{60}$.
This was done by the \textsc{Magma} function \texttt{IsIsomorphic}.
Hence, we have the following proposition.

\begin{prop}
For each $C$ of the codes $QR_{60}$, $P_{60}$ and $NV_{60}$, 
there is a four-negacirculant code $D$ such that $C \cong D$.
\end{prop}

It is worthwhile to determine whether there is a new ternary extremal  four-negacirculant
self-dual code of length $60$.

\begin{rem}
Two ternary extremal self-dual codes 
$D_{60,1}$ and $D_{60,2}$ of length $60$ were constructed in~\cite{HHKK}.
We verified by \textsc{Magma}  that $D_{60,1} \cong QR_{60}$ and $D_{60,2} \cong P_{60}$.
This was also done by the \textsc{Magma} function \texttt{IsIsomorphic}.
\end{rem}

\begin{rem}
Recently, it has been shown in~\cite{BCV} that there are exactly three inequivalent
ternary extremal self-dual codes of length $60$ having an automorphism of order $29$.
On the other hand, since each of $QR_{60}$, $P_{60}$ and $NV_{60}$ has an automorphism 
of order $29$,  the three codes found in~\cite{BCV} are  $QR_{60}$, $P_{60}$ and $NV_{60}$. 
\end{rem}

\bigskip
\noindent
\textbf{Acknowledgments.}
This work was supported by JSPS KAKENHI Grant Numbers 19H01802, 20K03719
and 21K03350.



\end{document}